\documentclass[a4paper]{amsart}
\usepackage{amsfonts}
\usepackage{amssymb}
\usepackage{amsmath}
\usepackage{amsthm}
\usepackage{latexsym}
\usepackage{graphicx}
\usepackage{layout}
\usepackage[colorlinks=true,urlcolor=blue,citecolor=red,linkcolor=blue,linktocpage,pdfpagelabels,bookmarksnumbered,bookmarksopen]{hyperref}
\usepackage{mathtools}
\usepackage{braket}
\usepackage{enumerate}
\usepackage{tikz}
\usetikzlibrary{intersections}

\newtheorem{theorem}{Theorem}[section]

\newtheorem{proposition}[theorem]{Proposition}
\newtheorem{lemma}[theorem]{Lemma}
\newtheorem{corollary}[theorem]{Corollary}
\newtheorem{remark}[theorem]{Remark}

\newcommand{\sezione}[1]{\section{#1}\setcounter{equation}{0}}

\def\R{\mathbb{R}}

\def\di12{\mathcal{D}^{1,2}(\R^n)}

\def\l{{\lambda}}

\def\0l{_{0,\l}}
\def\1l{_{1,\l}}
\def\2l{_{2,\l}}
\def\3l{_{3,\l}}
\def\4l{_{4,\l}}

\def\Om{\Omega}

\def\sideremark#1{\ifvmode\leavevmode\fi\vadjust{\vbox to0pt{\vss
 \hbox to 0pt{\hskip\hsize\hskip1em
 \vbox{\hsize2.1cm\tiny\raggedright\pretolerance10000
  \noindent #1\hfill}\hss}\vbox to15pt{\vfil}\vss}}}%

\newcommand{\F}{\color{red}}

\newtheorem*{theorem*}{Theorem}
\renewcommand{\phi}{\varphi}
\renewcommand{\Om}{\Omega}
\newcommand{\tangvet}{\textsl{t}}
\newcommand{\orig}{\mathbf{0}}
\newcommand{\curv}{\mathfrak{K}}
\newcommand{\eps}{\varepsilon}
\DeclarePairedDelimiter{\abs}{\lvert}{\rvert}
\DeclarePairedDelimiter{\scal}{\langle}{\rangle}
\DeclareMathOperator{\hess}{Hess}

\begin{document}
\title[critical points]{Uniqueness of the critical point for semi-stable solution in $\R^2$}
\thanks{This work was supported by INDAM-GNAMPA}
\thanks{D. Mukherjee's research is supported by the Czech Science Foundation, project GJ19--14413Y. This work was started while D. Mukherjee was visiting Mathematics Department of the University of Rome "La Sapienza" supported by INDAM-GNAMPA }

\author[De Regibus]{Fabio De Regibus}
\address{Dipartimento di Matematica, Universit\`a di Roma ``La Sapienza", P.le A. Moro 2 - 00185 Roma, Italy, e-mail: {\sf fabio.deregibus@uniroma1.it}.}
\author[Grossi]{Massimo Grossi }
\address{Dipartimento di Matematica, Universit\`a di Roma ``La Sapienza", P.le A. Moro 2 - 00185 Roma, Italy, e-mail: {\sf massimo.grossi@uniroma1.it}.}
\author[Mukherjee]{Debangana Mukherjee}
\address{Department of Mathematics and Statistics, Masaryk University, Kotlářská 267/2, 611 37 Brno, Czech Republic, e-mail: {\sf mukherjeed@math.muni.cz}.}

\maketitle
\begin{abstract}
In this paper we show the uniqueness of the critical point for \emph{semi-stable} solutions of the problem
\begin{equation*}
\begin{cases}
-\Delta u=f(u)&\text{in }\Omega\\
u>0&\text{in }\Omega\\
u=0&\text{on }\partial\Omega,
\end{cases}
\end{equation*}
where $\Om\subset\R^2$ is a smooth bounded domain whose boundary has \emph{nonnegative} curvature and $f(0)\ge0$. It extends a result by Cabr\'e-Chanillo to the case where the curvature of $\partial\Om$ vanishes.
\end{abstract}

\sezione{Introduction and statement of the main results}\label{s0}
In this paper we study the number of critical points of solutions of the problem
\begin{equation}\label{pb1}
\begin{cases}
-\Delta u=f(u)&\text{ in }\ \Om\\
u>0&\text{ in }\ \Om\\
u=0&\text{ on }\ \partial\Om
\end{cases}
\end{equation}
where $\Om\subset\R^2$ is a smooth bounded domain and $f$ is a smooth nonlinearity.\\
This is a classical problem of partial differential equations and many mathematicians gave important contributions. Since it is impossible to give an exhaustive list of references  we will limit ourselves to recall some results that are closer to the interest of this paper.\\
One of the first results concerns $f(u)=\l u$, hence $u$ is the first eigenfunction of the Laplacian with zero Dirichlet boundary condition. Here it was proved by Brascamp and Lieb~\cite{bl} and Acker, Payne and Philippin~\cite{app} in dimension $n=2$ that if $\Om\subset\R^n$ is strictly convex then $u$ admits a unique critical point in $\Omega$.\\
 A second seminal result is the fundamental theorem by Gidas, Ni and Nirenberg~\cite{gnn}.
\begin{theorem*}[\bf Gidas, Ni, Nirenberg]
 Let $\Om\subset\R^n$ be a  smooth bounded domain which is symmetric with respect to the plane $x_i=0$ for any $i=1,\dots,n$ and  convex with respect to any direction $x_1,\dots,x_n$. Suppose that $u$ is a positive
 solution to~\eqref{pb1}
where $f$ is a locally Lipschitz nonlinearity. Then 
\begin{itemize}
\item $u$ is symmetric with respect to $x_1,\dots,x_n$. (Symmetry)
\item  $\frac {\partial u}{\partial x_i}<0$ for $x_i>0$ and $i=1,\dots,n$. (Monotonicity)
\end{itemize}
\end{theorem*}
\noindent
Some conjectures claim that the uniqueness of the critical point holds in more general convex domains. A complete result for general nonlinearities is not known, there are some contributions for suitable cases (see~\cite{gm} and~\cite{gg1,gg2}).\\
Next we mention another important result which avoids the symmetry assumption on $\Om$ but assume that $u$ is  a semi-stable solution. We recall that $u$ is a \emph{semi-stable} solution of the problem~\eqref{pb1} if the linearized operator at $u$ is nonnegative definite, i.e. if for all $\phi\in\mathcal C^{\infty}_{0}(\Omega)$ one has
\[
\int_{\Omega}|\nabla\phi|^{2}-\int_{\Omega} f'(u)|\phi|^{2}\ge0,
\]
or equivalently if the first eigenvalue of the linearized operator $-\Delta-f'(u)$ in $\Omega$ is nonnegative.
\begin{theorem}[\bf Cabr\'e, Chanillo \cite{cc}]\label{i1}
 Assume $\Omega$ is a smooth, bounded and convex domain of $\R^2$ whose boundary has positive curvature. Suppose $f\ge0$ and $u$ is a semi-stable positive solution to~\eqref{pb1}. 
Then $u$ has an unique nondegenerate critical point.
\end{theorem}
In this paper we want to extend Cabr\'e-Chanillo's result allowing the curvature of $\partial\Om$ to vanish somewhere. Our main result is the following,
\begin{theorem}\label{mainthm}
Let $\Omega\subset\R^{2}$ be a smooth bounded convex domain whose boundary has nonnegative curvature and such that the subset of zero-curvature consists of isolated points or segments. Moreover we require that $f(0)\ge0$.\\
 If $u$ is a semistable solution to~\eqref{pb1} then $u$ has an unique critical point $x_0$. Moreover $x_0$ is a nondegenerate maximum point of $u$.
\end{theorem}
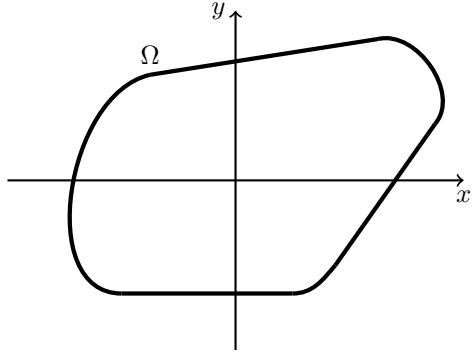
\begin{figure}[h]
\begin{tikzpicture}
	\draw[thick,->] (-3,0) -- (3,0) node[below] {$x$};
	\draw[thick,->] (0,-2.25) -- (0,2.25) node[left] {$y$};
	
	\coordinate (A) at  (-1.125,1.4);
	\coordinate (B) at  (-1.5,-1.5);
	\coordinate (C) at  (0.75,-1.5);
	\coordinate (D) at  (1.3125,-1.125);
	\coordinate (E) at  (2.625,0.75);
	\coordinate (F) at  (1.875,1.875);

	\draw [ultra thick] (A) to [out=193,in=180] (B);
	\draw [ultra thick] (B) to (C);
	\draw [ultra thick] (C) to [out=0,in=230] (D);
	\draw [ultra thick] (D) to (E);
	\draw [ultra thick] (E) to [out=50,in=13] (F);
	\draw [ultra thick] (F) to (A)node[above] {$\Omega$};
\end{tikzpicture}
\caption{Example of domain for which Theorem~\ref{mainthm} holds.}
\end{figure}

We point out that the extension of Theorem~\ref{i1} when the curvature of $\partial\Om$ vanishes somewhere is nontrivial. Indeed the proof in Theorem~\ref{i1} does not work in this case because the vector field $Z$ considered at page $7$  in~\cite{cc} is not well defined as the curvature of $\partial\Om$ vanishes.\\
Our main idea is to consider the vector field  $T:\overline\Omega\to\R^{2}$ (already used in~\cite{cc}) given by
\begin{equation}\label{i2}
T(q)=(u_{yy}(q)u_{x}(q)-u_{xy}(q)u_{y}(q),u_{xx}(q)u_{y}(q)-u_{xy}(q)u_{x}(q)).
\end{equation}
One of the main tool of our proof is to compute the \emph{topological degree} $\deg(T,\Om,\orig)$ of $T$. In particular if the curvature of $\partial\Om$ is positive then we have that (Lemma~\ref{sez2:lemma1})
\[
\deg(T,\Om, \orig)=1.
\]
A deeper analysis of the degree of $T$ concerns the $index$ of the zeros of $T$. It is not difficult to see (Lemma~\ref{b1}) that if $T(q)=0$ then either $q$ is a critical point of $u$ or $q$ is a critical point of the directional derivative 
\[
\frac{\partial u}{\partial v}=u_y(q)u_x-u_x(q)u_y.
\]
In Lemma~\ref{sez2:lemma2} we will compute the index in both cases. This is one of the crucial steps of the proof. We remark that this approach provides a quicker proof of Theorem~\ref{i1} because it simplifies the topological approach at pages $6-8$ in~\cite{cc}.\\
Eventually a careful analysis of the critical points of $\frac{\partial u}{\partial v}$ on $\partial\Om$ and the nodal lines of $\frac{\partial u}{\partial v}$ in $\Om$ ends the proof,\\
We stress that our result is sharp in the sense that it is possible to construct a bounded domain $\Om\subset\R^2$ and $u$ verifying
\begin{equation}
\begin{cases}
-\Delta u=1&\text{in }\Omega\\
u=0&\text{on }\partial\Omega,
\end{cases}
\end{equation}
such that $u$ admits  an arbitrarily number of critical points. Here $\partial\Om$ admits points with negative curvature but $\Om$  is ``close'' to a convex domain (see~\cite{gg3} for a precise statement).

The paper is organized as follows: in Section~\ref{s1} we recall some preliminary results (basically proved in~\cite{cc}) and prove the main properties of the vector field $T$. In Section~\ref{s2} we prove Theorem~\ref{mainthm}.

\sezione{Preliminary results}\label{s1}

Let $\Omega\subset\R^{2}$ be a smooth bounded domain and let $u$ be a solution to~\eqref{pb1} where $f\in\mathcal{C}^{1,\alpha}(\R^+,\R^+)$ with $\alpha\in(0,1]$. Recall that $u\in\mathcal C^{3}(\overline\Omega)$ by the standard regularity theory.

As in~\cite{cc}, we introduce the following notation: for every $\theta\in[0,2\pi)$ we write $e_{\theta}=(\cos\theta,\sin\theta)$ and we set
\begin{align*}
u_{\theta}&:=\scal{\nabla u, e_{\theta}}=\frac{\partial u}{\partial e_{\theta}},\\
N_{\theta}&:=\set{p\in\overline{\Omega}|u_{\theta}(p)=0},\\
M_{\theta}&:=\set{p\in N_{\theta}|\nabla u_{\theta}(p)=\orig}.
\end{align*}
Moreover, if the set $\set{u=c}$ is smooth then its curvature is given by
\begin{equation*}
\curv:=-\frac{u_{yy}u_{x}^2-2u_{xy}u_{x}u_{y}+u_{xx}u_{y}^2}{|\nabla u|^3}.
\end{equation*}

The following result tells us that the nodal sets of a semi-stable solution of~\eqref{pb1} are smooth curves without self intersection and every critical point of $u$ is nondegenerate.
\begin{proposition}
\label{sez2:prop1}
Let $\Omega\subset\R^{2}$ be a smooth bounded convex domain whose boundary has nonnegative curvature and such that the curvature vanishes only at isolated points. Assume that $f(0)\ge0$. If $u$ is a semi-stable solution of~\eqref{pb1} then for every $\theta\in[0,2\pi)$, the nodal set $N_{\theta}$ is a smooth curve in $\overline\Omega$ without self-intersection  which hits $\partial\Omega$ at the two end points of $N_{\theta}$. Moreover in any critical point of $u$ the Hessian has rank $2$.
\end{proposition}

\begin{proof}
The proof is given in~\cite{cc} in the case of positive curvature. For reader's convenience first we report the key steps of the proof in~\cite{cc} and next we add the case of zero curvature at isolated points of the boundary.

For any $\theta\in[0,2\pi)$ we have that
\begin{itemize}
\item Around any $p\in (N_{\theta}\cap\Omega)\setminus M_{\theta}$ the nodal set $N_{\theta}$ is a smooth curve.

\item If $p\in M_{\theta}\cap\Omega$, then $N_{\theta}$ consists of at least two smooth curves intersecting transversally at $p$.

\item There is no nonempty domain $H\subset\Omega$ such that $\partial H\subset N_{\theta}$ (where the boundary of $H$ is considered as a subset of $\R^{2}$).

\item If $p_i\in N_{\theta}\cap\partial\Omega$ by the implicit function theorem one has that around $p_{i}$, $N_{\theta}$ is a smooth curve intersecting $\partial\Omega$ transversally  in $p_{i}$. 


\item If $N_{\theta}\cap\partial\Omega=\{p_1,p_2\}$ then $M_\theta=\emptyset$  and any critical point of $u$ verifies that $\hess_u(p)$ has rank $2$.
\end{itemize}
We stress that all the above properties hold for semi-stable solutions $u$ in \emph{any} domain $\Om$.

Now we consider the case where the curvature of the boundary vanishes at isolated points. By the compactness of $\partial\Omega$ and the smoothness of $u$ we have that the curvature $\curv$ vanishes only at finitely many points of $\partial\Omega$.

Assume $\curv(p_{1})=0$ and $\curv(p_{2})>0$. If there exists $\rho>0$ such that $N_{\theta}\cap B_{\rho}(p_{1})\cap\Omega=\emptyset$ then the nodal curve $N_{\theta}$ starting from $p_{2}$ has to enclose a nonempty domain $H\subset\Omega$ with $\partial H\subset N_{\theta}$, but this yields to a contradiction. Otherwise $N_{\theta}$ consists of at least one curve starting from $p_{1}$ and disjoint from $\partial\Omega$ (since $N_{\theta}\cap\partial\Omega=\{p_{1},p_{2}\}$). If there are more then one curve this implies again that there exists a subdomain $H$ as before and this is a contradiction. Hence as in~\cite{cc}, around $p_{1}$ we have that $N_{\theta}$ is a smooth curve intersecting $\partial\Omega$ in $p_{1}$.

If $\curv(p_{1})=\curv(p_{2})=0$ we can argue similarly to get that around each $p_{i}$, $N_{\theta}$ is a smooth curve intersecting $\partial\Omega$ in $p_{i}$, for $i=1,2$. Note that if there exists $\rho>0$ such that $N_{\theta}\cap B_{\rho}(p_{1})\cap\Omega=N_{\theta}\cap B_{\rho}(p_{2})\cap\Omega=\emptyset$ then, since $N_{\theta}\cap\Omega$ is nonempty by the fact that there exists at least a critical point of $u$ in $\Omega$,  the nodal set enclose again a domain as before.

The remaining claims of the proposition follow arguing as in~\cite{cc}.
\end{proof}

For $u$ solution of~\eqref{pb1}, consider the map $T:\overline\Omega\to\R^{2}$ given by
\[
T(q)=(u_{yy}(q)u_{x}(q)-u_{xy}(q)u_{y}(q),u_{xx}(q)u_{y}(q)-u_{xy}(q)u_{x}(q)).
\]
Since $u\in\mathcal{C}^{3}(\overline{\Omega})$, $T$ is of class $\mathcal{C}^{1}$. In next lemmata we state some important properties of the vector field $T$.
\begin{lemma}
\label{sez2:lemma1}
Let $D\subset\Omega$ be a smooth convex domain such that $\partial D$ has positive curvature. Then $\deg(D,T,\orig)=1$.
\end{lemma}
\begin{proof}
Let $p=(x_{p},y_{p})\in\Omega$ and consider the homotopy
\begin{align*}
H:[0,1]\times\overline{\Omega}&\to\R^{2}\\
(t,q)&\mapsto tT(q)+(1-t)(q-p).
\end{align*}
$H$ is an admissible homotopy, i.e. $H\big(t,(x,y)\big)\ne\orig$ for any $t\in[0,1]$ and $(x,y)\in\partial\Om$.
 Otherwise, there would exist $\tau\in[0,1]$ and $\overline{q}=(\overline{x},\overline{y})\in\partial \Omega$ such that $H(\tau,\overline{q})=0$, i.e.
\begin{equation}
\label{sez2:lemma1:eq1}
\begin{cases}
\tau(u_{yy}(\overline{q})u_{x}(\overline{q})-u_{xy}(\overline{q})u_{y}(\overline{q}))=(\tau-1)(\overline{x}-x_{p})\\
\tau(u_{xx}(\overline{q})u_{y}(\overline{q})-u_{xy}(\overline{q})u_{x}(\overline{q}))=(\tau-1)(\overline{y}-y_{p}).
\end{cases}
\end{equation}
Then, multiplying the first equation by $u_{x}(\overline{q})$, the second by $u_{y}(\overline{q})$ and summing we get
\begin{align*}
\tau(u_{yy}(\overline{q})u_{x}^2(\overline{q})-2u_{xy}(\overline{q})u_{x}(\overline{q})u_{y}&(\overline{q})+u_{xx}(\overline{q})u_{y}^2(\overline{q}))\\
	&=(\tau-1)[(\overline{x}-x_{p})u_{x}(\overline{q})+(\overline{y}-y_{p})u_{y}(\overline{q})],
\end{align*}
and writing $\nu=(\nu_{x},\nu_{y})$ for the unit normal exterior vector in $\overline{q}$, it follows
\begin{equation}
\label{sez2:lemma1:eq2}
-\tau \curv(\overline{q})\abs{\nabla u(\overline{q})}^3=(\tau-1)u_{\nu}(\overline{q})[(\overline{x}-x_{p})\nu_{x}+(\overline{y}-y_{p})\nu_{y}].
\end{equation}
Since $\Omega$ is strictly star-shaped with respect to the point $p$ we have $(\overline{x}-x_{p})\nu_{x}+(\overline{y}-y_{p})\nu_{y}>0$. Since $\curv(\overline{q})>0$ and $u_{\nu}(\overline{q})<0$ by~\eqref{sez2:lemma1:eq2} we get $\tau=0$ and thanks to~\eqref{sez2:lemma1:eq1} this yields to $\overline{q}=p$ which is clearly a contradiction.

Then $H$ is an admissible homotopy and so we conclude
\[
\deg(\Omega,T,\orig)=\deg(\Omega,Id-p,\orig)=1.\qedhere
\]
\end{proof}

\begin{lemma}
\label{b1}
If $q\in\Omega$ is such that $T(q)=\orig$ then either
\begin{equation*}
q\hbox{ is a critical point for $u$},
\end{equation*}
or
\begin{equation*}
\hbox{ there exists $\theta\in[0,2\pi)$ such that $q\in M_{\theta}$}.
\end{equation*}
\end{lemma}
\begin{proof}
Of course if $q$ is a critical point for $u$ then $T(q)=\orig$. So suppose that $q$ is not a critical point for $u$ and consider $\theta\in[0,2\pi)$ such that $(\cos\theta,\sin\theta)=\left(\frac{u_y(q)}
{\sqrt{u_x^2(q)+u_y^2(q)}},-\frac{u_x(q)}{\sqrt{u_x^2(q)+u_y^2(q)}}\right)$. Then it is straightforward to verify that
\[
u_\theta=\cos\theta u_x+\sin\theta u_y
\]
satisfies $u_\theta(q)=0$ and $\nabla u_\theta(q) =\orig$. Hence $q\in M_{\theta}$.
\end{proof}

\begin{remark}
\label{sez2:rmk1}
We point out that if $q\in M_{\theta}$ then up to a rotation we can assume that 
\begin{equation}\label{b2}
u_{x}(q)=u_{xx}(q)=u_{xy}(q)=0.
\end{equation}
\end{remark}
From now if $q$ is an isolated zero of $T$, for $r>0$ small enough, we denote by $ind(T,q)=\deg\big(T,B(q,r),\orig\big)$. 
\begin{lemma}
\label{sez2:lemma2}
Let $q\in\Omega$ be such that $T(q)=\orig$. Then we have that
\begin{itemize}
\item If  $q$ is a nondegenerate critical point for $u$, then $ind(T,q)=1$.
\item If $q\in M_\theta$ for some $\theta\in[0,2\pi)$ and it is a nondegenerate critical point for $u_\theta$  then $ind(T,q)=-1$.
\end{itemize}
\end{lemma}
\begin{proof}
One has
\[
\mathrm{Jac}_{T}=
\begin{pmatrix}
u_{xx}u_{yy}-u_{xy}^2+u_{x}u_{xyy}-u_{y}u_{xxy} & u_{x}u_{yyy}-u_{y}u_{xyy}\\
u_{y}u_{xxx}-u_{x}u_{xxy} &u_{xx}u_{yy}-u_{xy}^2+u_{y}u_{xxy}-u_{x}u_{xyy}
\end{pmatrix}.
\]
If $q$ is a critical point for $u$ we have 
\[
\det\mathrm{Jac}_{T}(q)=(\det\hess_{u}(q))^{2},
\]
and since it is nondegenerate we get that $ind(T,q)=1$.\\
On the other hand if $q\in M_\theta$ by Remark~\ref{sez2:rmk1} we have that~\eqref{b2} holds and then
\begin{align*}
\det\mathrm{Jac}_{T}(q)&=-u_{y}^{2}(q)(u_{xxy}^{2}(q)-u_{xxx}(q)u_{xyy}(q))\\
	&=-u_{y}^{2}(q)(u_{xxy}^{2}(q)+u_{xxx}^{2}(q)),
\end{align*}
where the last equality follows differentiating~\eqref{pb1}. Finally the nondegeneracy of $q$ for $\nabla u_\theta$ gives that $u_{xxy}^{2}(q)+u_{xxx}^{2}(q)\ne0$ and the claim follows.
\end{proof}

As remarked in the introduction, the previous lemma gives a simplified proof of Cabr\'e-Chanillo's result.
\begin{proof}[Proof of Theorem~\ref{i1}]
By Proposition~\ref{sez2:prop1} we have that $M_\theta=\emptyset$ for any $\theta\in[0,2\pi)$. Hence if $T(q)=0$ then $q$ is a critical point of $u$. Moreover it is nondegenerate (otherwise $u\in M_\theta$ for some $\theta\in[0,2\pi)$). 

Finally by Lemma~\ref{sez2:lemma1} and Lemma~\ref{sez2:lemma2} we have
\[
1=\deg(T,\Om,\orig)=\sum_{\hbox{$q$ such that $\nabla u=0$}}\!\!\!\!\!\!ind(T,q)=\sharp\{\hbox{critical points of }u\},
\]
which gives the claim.
\end{proof}
\begin{corollary}
\label{sez2:cor1}
Let $D\subset\overline\Omega$ be such that $M_{\theta}\cap D=\emptyset$ for all $\theta\in[0,2\pi)$ and $\orig\not\in T(\partial D)$. If $\deg(D,T,\orig)=1$, then $u$ has exactly one critical point in $D$ which is a maximum with negative definite Hessian.
\end{corollary}
\begin{proof}
Considering that $M_{\theta}\cap D=\emptyset$ for all $\theta\in[0,2\pi)$ implies that the Hessian of $u$ has maximal rank in every critical point in $D$, the proof easily follows from Remark~\ref{sez2:rmk1} and Lemma~\ref{sez2:lemma2}.
\end{proof}
We end this section by proving a technical result which will be useful later. The proof is essentially the one in~\cite{pm}.

\begin{lemma}
\label{sez2:lemma3}
Let $\Omega\subset\set{(x,y)\in\R^2|y<0}$ be a bounded smooth domain such that $\partial\Omega$ is tangent to the $x$-axis in $\orig$ and consider $F\in\mathcal{C}^2(\overline{\Omega},\R)$ with
\[
F(\orig)=F_{x}(\orig)=F_{y}(\orig)=F_{xy}(\orig)=0,\quad F_{xx}(\orig)<0,\quad\text{and}\quad F_{yy}(\orig)>0.
\]
Then there exist $\delta,\eta>0$ and two functions $Y_{1}\in\mathcal{C}^{1}([-\delta,0],[-\eta,0])$ and $Y_{2}\in\mathcal{C}^{1}([0,\delta],[-\eta,0])$ such that
\begin{enumerate}[(i)]
\item $Y_{1}(0)=Y_{2}(0)=0$,
\item $Y_{1}'(0)=\sqrt{-\frac{F_{xx}(\orig)}{F_{yy}(\orig)}}$, $Y_{2}'(0)=-\sqrt{-\frac{F_{xx}(\orig)}{F_{yy}(\orig)}}$,
\item $F(x,y)=0$ if and only if $y=Y_{1}(x)$ for $x\in[-\delta,0]$ and $y=Y_{2}(x)$ for $x\in[0,\delta]$.
\end{enumerate}
\end{lemma}
\begin{proof}
Since $F_{xx}(\orig)<0$ and $F_{yy}(\orig)>0$, there exists $\eta>0$ such that
\[
F_{xx}(x,y)<0\quad\text{and}\quad F_{yy}(x,y)>0,\quad\text{for }(x,y)\in\overline{\Omega}\cap[-\eta,\eta]\times[-\eta,0].
\]
From now on we work with $(x,y)\in\overline{\Omega}\cap[-\eta,\eta]\times[-\eta,0]$. Since $F_{\tangvet}(\orig)=F_{x}(\orig)=0$ and $F_{\tangvet\tangvet}(\orig)=F_{xx}(\orig)<0$ it follows that
\[
F(x,y)<0,\quad\text{for }(x,y)\in\partial\Omega\text{ with }0<\abs{x}\le\eta',
\]
where $0<\eta'\le\eta$. Moreover, by the strictly convexity of the function
\[
y\mapsto F(x,y),\quad\text{for }x\text{ fixed},
\]
and the fact that $F_{y}(\orig)=0$, one has
\[
F(0,y)>0,\quad\text{for }y\not=0.
\]
In particular $F(0,-\eta')>0$ and by continuity $F(x,-\eta')>0$ for $-\delta<x<\delta$, where $0<\delta\le\eta'$. Now, for $-\delta<x\le0$ the function
\[
y\mapsto F(x,y),
\]
is such that
\[
F(x,\overline{y})<0,\text{ for }(x,\overline{y})\in\partial\Omega\quad F(x,-\eta')>0,\quad F_{yy}(x,\cdot)>0,
\]
and since $\overline{y}>-\eta'$ then there exists an unique $y_x\in(\overline{y},-\eta')$ such that $F(x,y_x)=0$; furthermore, since $F_{yy}(x,y)>0$, one has $F_y(x,y_x)<0$. Then the zero set of $F$ is given by a continuous function $Y_{1}:[-\delta,0]\to[-\eta,0]$ (where the continuity in $0$ holds since we can choose $\eta$ arbitrarily small). From the implicit function theorem we have $Y_{1}\in\mathcal C^0([-\delta,0])\cap\mathcal C^1([-\delta,0))$ with
\begin{equation}
\label{sez2:lemma2:eq1}
F_{x}(x,Y_{1}(x))+F_{y}(x,Y_{1}(x))Y_{1}'(x)=0,\quad\text{for }x\not=0.
\end{equation}
Moreover, from
\begin{equation}
\label{sez2:lemma2:eq2}
F(x,y)=\frac{1}{2}(F_{xx}(\orig)+o(1))x^2+\frac{1}{2}(F_{yy}(\orig)+o(1))y^2,
\end{equation}
we deduce that $(x,Y_{1}(x))$ belongs to a cone around the line $y=-\frac{F_{xx}(\orig)}{F_{yy}(\orig)}x$ and it is derivable in $0$ with $Y_{1}'(0)=\sqrt{-\frac{F_{xx}(\orig)}{F_{yy}(\orig)}}$. Indeed, for $y=Y_{1}(x)$ in~\eqref{sez2:lemma2:eq2} we obtain
\[
Y_{1}(x)^2=-\frac{F_{xx}(\orig)}{F_{yy}(\orig)+o(1)}x^2+\frac{x^2}{F_{yy}(\orig)+o(1)}o(1),
\]
then 
\begin{equation}
\label{sez2:lemma2:eq3}
Y_{1}(x)=\left(\sqrt{-\frac{F_{xx}(\orig)}{F_{yy}(\orig)}}+o(1)\right)x,
\end{equation}
and for $x\to0$ we get the claim. Moreover, we have
\begin{align*}
F_{x}(x,y)&=F_{xx}(\orig)x+o(x+y),\\
F_{y}(x,y)&=F_{yy}(\orig)y+o(x+y),
\end{align*}
and then from~\eqref{sez2:lemma2:eq1} and~\eqref{sez2:lemma2:eq3} it follows
\begin{align*}
Y_{1}'(x)&=-\frac{F_{x}(x,Y_{1}(x))}{F_{y}(x,Y_{1}(x))}\\
	&=-\frac{F_{xx}(\orig)x+o(x)+o(Y_{1}(x))}{F_{yy}(\orig)Y_{1}(x)+o(x)+o(Y_{1}(x))}\\
	&=-\frac{F_{xx}(\orig)+o(1)}{F_{yy}(\orig)\sqrt{-\frac{F_{xx}(\orig)}{F_{yy}(\orig)}}+o(1)}=\sqrt{-\frac{F_{xx}(\orig)}{F_{yy}(\orig)}}+o(1)\quad\text{for }x\to0,\\
\end{align*}
that is $Y_{1}\in\mathcal C^1([-\delta,0])$. For $0\le x<\delta$ we can argue analogously.
\end{proof}

\begin{remark}
\label{sez2:rmk3}
A similar result can be proved for interior points of the domain, see~\cite{pm}.
\end{remark}

\sezione{Proof of main theorems}\label{s2}
In this section we prove the main results of the paper. First of all we fix the assumptions on the domain $\Om$.
\vskip0.2cm\noindent
\begin{equation}\label{c0}
\hbox{{\bf Assumption on $\Om$}}
\end{equation}
Suppose that $\Om$ is a convex domain such that the curvature is zero in a single point of his boundary and positive elsewhere. Up to a rotation and a translation we assume $\Omega\subset\set{(x,y)\in\R^2|y<0}$ such that $\partial\Omega$ is tangent to the $x$-axis in $\orig$, which is the only point where the curvature is zero. 

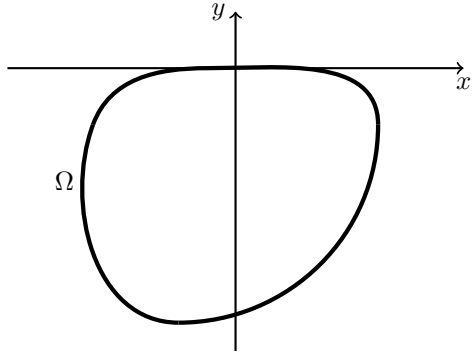
\begin{figure}[h]
\begin{tikzpicture}
	\draw[thick,->] (-3,0) -- (3,0) node[below] {$x$};
	\draw[thick,->] (0,-3.75) -- (0,0.75) node[left] {$y$};
	
	\coordinate (A) at  (-0.375,0);
	\coordinate (B) at  (-1.875,-0.75);
	\coordinate (C) at  (-0.75,-3.375);
	\coordinate (D) at  (1.875,-0.75);

	\draw [ultra thick] (A) to [out=180,in=70] (B);
	\draw [ultra thick] (B) to [out=250,in=180] (C);
	\draw [ultra thick] (C) to [out=0,in=270] (D);
	\draw [ultra thick] (D) to [out=90,in=0] (A);
	
	\node at (-2.25,-1.5) {$\Omega$};
	
\end{tikzpicture}
\caption{Example of domain which satisfies~\eqref{c0}.}
\end{figure}

\begin{theorem}
\label{sez3:thm1}
Suppose that $\Om$ satisfies~\eqref{c0}. If $u$ is a semi-stable solution of~\eqref{pb1} then $u$ has an unique critical point $x_0$. Moreover $x_0$ is the maximum of $u$ and it has negative definite Hessian.
\end{theorem}

To prove the theorem, we need the following auxiliary lemma.

\begin{lemma}
\label{sez3:lemma1}
Suppose that $\Om$ satisfies~\eqref{c0}, $u$ is a semi-stable solution of~\eqref{pb1} and $u_{xy}(\orig)=0$. Then $\curv_{y}(\orig)<0$.
\end{lemma}
\begin{proof}
From assumption $\curv(\orig)=0$ and from $u_x(\orig)=0$ we easily get that $u_{xx}(\orig)=0$. It follows that
\[
\curv_{y}(\orig)=-\frac{u_{xxy}(\orig)u_{y}^2(\orig)}{\abs*{\nabla u}^{3}}.
\]
We claim that
\begin{equation}\label{c1}
u_{xxy}(\orig)>0,
\end{equation}
which ends the proof since $u_{y}(\orig)=u_{\nu}(\orig)<0$ by the Hopf boundary lemma. In order to prove~\eqref{c1} we divide the proof in four steps.
\vskip0.2cm
\emph{Step 1: $u_{xxy}(\orig)\not=0$.}\\
Since $u_{y}(\orig)\not=0$ by the implicit function theorem we get that near the origin one has $u(x,y)=0$ if and only if $y=\phi(x)$, for some function $\phi$. In particular, by the assumptions on the boundary of $\Omega$, we have $\phi(0)=\phi'(0)=\phi''(0)=\phi'''(0)=0$ and differentiating $u(x,\phi(x))=0$ we deduce $u_{xxx}(\orig)=0$. Moreover, differentiating~\eqref{pb1}, we get that also $u_{xyy}(\orig)=0$. If $u_{xxy}(\orig)=0$, then the Taylor expansion of $u_x$ in a neighborhood of $\orig$ becomes
\[
u_x(x,y)=\text{homogeneous harmonic polynomial of order three}+O\big((x^2+y^2)^2\big),
\]
This means that locally $N_{0}=\set{u_x=0}$ consists of at least three branches of curves and at least two must be entering in $\Omega$, a contradiction with Proposition~\ref{sez2:prop1}.
\vskip0.2cm
\emph{Step 2: parametrization of $N_{0}$ near the origin.}\\
Let $F(x,y)=u_{x}(x,y)$ with $(x,y)\in\Omega$, then up to a rotation and eventually changing sign one has
\[
F(\orig)=F_{x}(\orig)=F_{y}(\orig)=F_{xy}(\orig)=0,\quad -F_{xx}(\orig)=F_{yy}(\orig)>0.
\]
Then we can apply Lemma~\ref{sez2:lemma3} and taking into account~\eqref{sez2:lemma2:eq3} and the fact that $u_{x}=0$ consist in exactly one branch entering in $\Omega$ from $\orig$, the nodal curve $N_{0}$ can be parametrized as
\[
\mathcal{C}=
\begin{cases}
x=g(t)\\
y=t
\end{cases}\quad t\in[-\delta,0],
\]
for some $\delta>0$ and $g(t)=o(t)$.
\vskip0.2cm
\emph{Step 3: $u_{xx}\big(g(t),t\big)\le0$ for $t$ close to $0$.}\\
Let $(\overline x,\overline y)\in\partial\Omega$ close to $\orig$ with $\overline x<0$ and $(g(\overline y),\overline y)\in\mathcal{C}$. Since $\overline x<0$, $u_x(x,\overline y)>0$ for $\overline{x}\le x<g(\overline y)$ and $u_x(g(\overline y),\overline y)=0$ we derive that  $u_{xx}(g(\overline y),\overline y)\le0$.
\vskip0.2cm
\emph{Step 4: end of the proof.}\\
Set $H(t):=u_{xx}\big(g(t),t\big)$  for $t\in[-\delta,0]$. By the previous step we have that $H(t)\le0$ and by the assumptions on $\Om$ one has $H(0)=0$. Hence
\[
H'(0)\ge0.
\]
Finally, $H'(t)=u_{xxx}\big(g(t),t\big)g'(t)+u_{xxy}\big(g(t),t\big)$ and so $0\le H'(0)=u_{xxy}\big(\orig)$ which gives the claim thanks to Step 1.
\end{proof}
\begin{proof}[Proof of Theorem~\ref{sez3:thm1}]
As remarked we have $u_x(\orig)=u_{xx}(\orig)=0$ and $u_{y}(\orig)<0$. We now distinguish the two cases.  according to whether $u_{xy}(\orig)$ vanishes or not.
\vskip0.2cm
\emph{Case 1: $u_{xy}(\orig)\not=0$.}\\
Similarly as in the proof of Lemma~\ref{sez2:lemma1} we consider the map $T:\Omega\to\R^{2}$ defined in~\eqref{i2} and the homotopy
\begin{align*}
H:[0,1]\times\overline{\Omega}&\to\R^{2}\\
(t,q)&\mapsto tT(q)+(1-t)(q-p),
\end{align*}
for some $p=(x_{p},y_{p})\in\Omega$. Let us show that $H$ is an admissible homotopy. Otherwise there would exist $\tau\in[0,1]$ and $\overline{q}=(\overline{x},\overline{y})\in\partial \Omega$ such that~\eqref{sez2:lemma1:eq1} and~\eqref{sez2:lemma1:eq2} hold. Then we deduce $\curv(\overline{q})=0$ and $\tau=1$, which thanks to the first equation of~\eqref{sez2:lemma1:eq1} yields to $-u_{y}(\orig)u_{xy}(\orig)=0$: a contradiction.
Then $H$ is an admissible homotopy and so
\[
\deg(\Omega,T,\orig)=\deg(\Omega,Id-p,\orig)=1.
\]
The claim follows from Corollary~\ref{sez2:cor1}.
\vskip0.2cm
\emph{Case 2: $u_{xy}(\orig)=0$.}\\
This case is more delicate because $T(\orig)=\orig$ and since $\orig\in\partial\Omega$ the degree of $T$ is not well defined. For this reason we introduce $\Omega_{\eps}:=\Omega\setminus\overline{B}_{\eps}$ where $\eps>0$ is chosen such that
\begin{equation}\label{sez3:thm1:eq1}	
\abs{\nabla u}\not=0,\quad\text{in }\overline{B}_{\eps}\cap\overline\Omega,
\end{equation}
and
\begin{equation}\label{sez3:thm1:eq2}
\Omega_{\eps}\text{ is star-shaped with respect to some }p=(x_{p},y_{p}),
\end{equation}
(such an $\eps$ exists by the Hopf boundary lemma). Now, consider again the map $T\in\mathcal C^1(\overline{\Omega}_\eps,\R^{2})$. In this way the degree of $T$ is well defined and if the homotopy
\begin{align}
\label{sez3:thm1:eq3}H_\eps:[0,1]\times\overline{\Omega}_{\eps}&\to\R^{2}\\
\label{sez3:thm1:eq4}(t,q)&\mapsto tT(q)+(1-t)(q-p),
\end{align}
is admissible then we deduce
\[
\deg(\Omega_{\eps},T,\orig)=1.
\]
Assume, by contradiction, that the homotopy $H_{\eps}$ is not admissible. Hence, there exist $\tau_{\eps}\in[0,1]$ and $q_{\eps}=(x_{\eps},y_{\eps})\in\partial \Omega_{\eps}$ such that $H_{\eps}(\tau_{\eps},q_{\eps})=0$, i.e.
\begin{equation}
\label{sez3:thm1:eq5}
\begin{cases}
\tau_{\eps}(u_{yy}(q_{\eps})u_{x}(q_{\eps})-u_{xy}(q_{\eps})u_{y}(q_{\eps}))=(\tau_{\eps}-1)(x_{\eps}-x_{p})\\
\tau_{\eps}(u_{xx}(q_{\eps})u_{y}(q_{\eps})-u_{xy}(q_{\eps})u_{x}(q_{\eps}))=(\tau_{\eps}-1)(y_{\eps}-y_{p}).
\end{cases}
\end{equation}
Proceeding as in the previous step we get
\begin{equation}
\label{b3}
-\tau_{\eps}\curv(q_{\eps})\abs{\nabla u(q_{\eps}}^3=(\tau_{\eps}-1)[(x_{\eps}-x_{p})u_{x}(q_{\eps})+(y_{\eps}-y_{p})u_{y}(q_{\eps})].
\end{equation}
Since for all $q=(x,y)\in\partial\Omega$, writing again $\nu=(\nu_{x},\nu_{y})$ for the unit normal exterior vector in $q$, we have
\[
(x-x_{p})u_{x}(q)+(y-y_{p})u_{y}(q)=u_{\nu}(q)[(x-x_{p})\nu_{x}+(y-y_{p})\nu_{y}]<0,
\]
by continuity it follows that $(x_{\eps}-x_{p})u_{x}(q_{\eps})+(y_{\eps}-y_{p})u_{y}(q_{\eps})<0$ for all $q_{\eps}\in\partial\Omega_{\eps}$.
Since $\curv>0$ on $\partial\Omega\setminus\{\orig\}$, from~\eqref{b3} it follows that $q_{\eps}\in\partial\overline{B}_{\eps}\cap\Omega$ and $\curv(q_{\eps})\le0$. Then the vertical line $x=x_{\eps}$ hits $\partial\Omega$ in an unique point $(x_{\eps},y(x_{\eps}))$, with $y(x_{\eps})> y_{\eps}$. Since $\curv(x_{\eps},y(x_{\eps}))\ge0$, there exists $p_{\eps}\in B_{\eps}\cap\Omega$ such that $\curv_{y}(p_{\eps})\ge0$ and for $\eps\to0$ we have $\curv_{y}(\orig)\ge0$, but this is in contradiction with Lemma~\ref{sez3:lemma1}.

Since $\deg(\Omega_{\eps},T,\orig)=1$ it is possible to apply Corollary~\ref{sez2:cor1} to get that there exists exactly one critical point in $\Omega_{\eps}$ (a maximum with negative definite Hessian). Moreover by the assumptions on $\eps$ that there are no critical points in $\overline{B}_{\eps}\cap\Omega$ and the claim follows.
\end{proof}

We now treat domains where the curvature vanishes at a segment of its boundary.

\begin{figure}[h]
\begin{tikzpicture}
	\draw[thick,->] (-3,0) -- (3,0) node[below] {$x$};
	\draw[thick,->] (0,-3.75) -- (0,0.75) node[left] {$y$};
	
	\coordinate (A) at  (-1.125,0);
	\coordinate (B) at  (0.375,-3.375);
	\coordinate (C) at  (1.875,0);

	\draw [ultra thick] (A) to [out=180,in=180] (B);
	\draw [ultra thick] (B) to [out=0,in=0] (C);
	\draw [ultra thick] (C) to (A);
	
	\node at (0.75,0)[above] {$\Gamma$};
	\node at (-1.5,-2.25) {$\Omega$};
	
\end{tikzpicture}
\caption{Example of domain for which Theorem~\ref{sez3:thm2} holds.}
\end{figure}
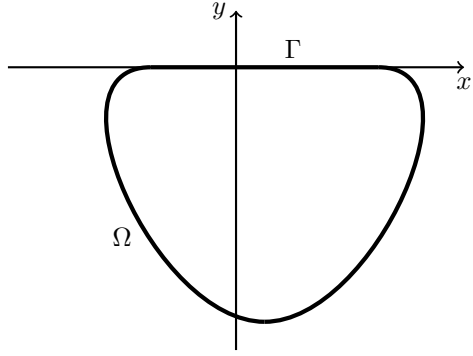

\begin{theorem}
\label{sez3:thm2}
Let $\Omega\subset\set{(x,y)\in\R^2|y<0}$ be a smooth open bounded, convex domain, whose boundary has nonnegative curvature and such that the zero curvature set is an interval $\Gamma$ on the $x$-axis. If $u$ is a semi-stable solution of~\eqref{pb1} then $u$ has an unique critical point $x_0$. Moreover $x_0$ is the maximum of $u$ and it has negative definite Hessian.
\end{theorem}
\begin{proof}
We want to argue as in Proposition~\ref{sez2:prop1} to show that for every $\theta\in(0,\pi)\cap(\pi,2\pi)$, the nodal set $N_{\theta}$ is a smooth curve in $\overline\Omega$ homeomorphic to the closed interval $[0,1]$ (without self-intersection) which intersects $\partial\Omega$ at the two end points of $N_{\theta}$ and $M_\theta=\emptyset$.\\
First we recall that there exists an unique point $p\in\partial\Omega\setminus\Gamma$ such that $u_x(p)=0$ and in a neighborhood of $p$, $N_{0}$ is a smooth curve that intesects $\partial\Omega$ transversally  at $p$.

Next we claim that there exists at least a point $\xi\in\Gamma$ such that around that point $N_{0}$ consists exactly of $2$ branches of curves: the first is $\Gamma$ while the second intesects $\partial\Omega$ at $\xi$. Indeed if there exists $\eps>0$ such that $\Omega\cap N_{0}\cap(\R\times(-\eps,0])=\emptyset$ then  the nodal curve $N_{0}$ starting from $p$ has to enclose a nonempty domain $H\subset\Omega$ with $\partial H\subset N_{0}$, but this yields to a contradiction by Proposition~\ref{sez2:prop1}. Analogously we get that we cannot have more than one branch of $N_{0}$ exiting by $\xi$ otherwise we create again such a domain $H\subset\Om$.

Finally we have the uniqueness of  $\xi\in\Gamma$ such that $N_{0}$ consists of one curve starting from $\xi$ and disjoint from $\partial\Omega$. Indeed if there exists another point $\eta\in\Gamma$ with the same property we can argue as before to get the existence of $H\subset\Om$ which yields a contradiction.

So we have proved  that $N_{0}$ is the union of two smooth curves in $\overline\Omega$: one is   the subset of the boundary $\Gamma$  and the other is homeomorphic to the closed interval $[0,1]$ (without self-intersection) and intersects $\partial\Omega$ at the two end points. They intersect each other only in the point $\xi$. Moreover we have $M_\theta=\emptyset$ and in any critical point $q\in\Omega$ of $u$ the Hessian has rank $2$ (same argument of Proposition~\ref{sez2:prop1}).

Up to a translation we can assume $\xi=\orig$.
We point out that for all $q\in\Gamma$, by the Hopf boundary lemma and $\curv=0$, it holds
\[
u_{x}=0,\quad u_{xx}=0,\quad u_{y}<0.
\]
Moreover, if $q\not=\orig$ one has $u_{xy}\not=0$ otherwise, locally, $u_x$ is an harmonic polynomial of degree at least $2$ and this implies that there exists a branch of $N_{0}$ entering in $\Om$, a contradiction with the uniqueness of the point $\xi=\orig$ with this property.

As in the proof of Case $2$ of Theorem~\ref{sez3:thm1}, let $\eps>0$ such that~\eqref{sez3:thm1:eq1} and~\eqref{sez3:thm1:eq2} are verified. Furthermore, if the homotopy defined in~\eqref{sez3:thm1:eq3}~-~\eqref{sez3:thm1:eq4} is not admissible, then~\eqref{sez3:thm1:eq5} and~\eqref{b3} still hold for some $\tau_{\eps}\in[0,1]$ and $q_{\eps}=(x_{\eps},y_{\eps})\in\partial \Omega_{\eps}$.

Let us prove that $u_{xxy}(\orig)>0$ (this implies that $\curv_{y}(\orig)=-\frac{u^2_{y}(\orig)u_{xxy}(\orig)}{\abs*{\nabla u}^{3}}<0$). Indeed, since $u_{xy}\ne0$ on $\Gamma\setminus\{\orig\}$ we have $u_{xy}>0$ on $\set{(x,y)\in\Gamma|x<0}$  and $u_{xy}<0$ on $\set{(x,y)\in\Gamma|x>0}$.  It follows $u_{xxy}(\orig)\ge0$, but if equality holds we can argue as in the proof of Lemma~\ref{sez3:lemma1} to get a contradiction. 

Next we can repeat the same argument as in Case $2$ of the proof of Theorem~\ref{sez3:thm1}
getting that the homotopy is admissible and $\deg(\Omega_{\eps},T,\orig)=1$. Finally we apply  Corollary~\ref{sez2:cor1} to conclude as in the proof of Theorem~\ref{sez3:thm1}.
\end{proof}

\begin{remark}
We observe that if $\partial\Omega$ contains more then one component homeomorphic to an interval, then they are in a finite number: $I_1,\dots,I_m$. Moreover, since the domain is convex, they are parallel at most at pairs.
Then also in this case it is possible to prove that for every $\theta\in[0,2\pi)$, the nodal set $N_{\theta}$ is a smooth curve in $\overline\Omega$ homeomorphic to the closed interval $[0,1]$ (without self-intersection) which intersects $\partial\Omega$ at the two end points of $N_{\theta}$ and in any critical point $q\in\Omega$ of $u$ the Hessian has rank $2$.
\end{remark}

Finally the proof of Theorem~\ref{mainthm} easy follows.

\begin{proof}[Proof of Theorem~\ref{mainthm}]
Let $\mathcal{K}:=\set{p\in\partial\Omega|\curv(p)=0,\,\,u_{\nu\tangvet}(p)=0}$ where $\nu$ is the normal exterior unit vector and $\tangvet$ the tangent one. Then define
\[
\Omega_{\eps}:=\Omega\setminus\bigcup_{q\in\mathcal{K}}\overline{B_{\eps}(q)},
\]
with $\eps>0$ such that
\begin{gather*}
B_{2\eps}(q_{1})\cap B_{2\eps}(q_{2})=\emptyset,\quad\text{for all } q_{1},q_{2}\in\mathcal{K},\\
\Omega_{\eps}\text{ is star-shaped with respect to some }p,\\
\abs{\nabla u}\not=0,\quad\text{in }\bigcup_{q\in\mathcal{K}}\overline{B_{\eps}(q)}\cap\overline\Omega.
\end{gather*}
The proof follows arguing as in the preceding theorems.
\end{proof}

\bibliography{DGMarxiv.bib}
\bibliographystyle{amsplain}
\end{document}